\documentclass[11pt,english]{smfart}

\usepackage{amsmath,leftidx,amsthm}
\usepackage[all]{xy}
\usepackage{xspace,smfthm}
\usepackage[psamsfonts]{amssymb}
\usepackage[latin1]{inputenc}
\usepackage{graphicx,color,mathrsfs}
\usepackage{hyperref,fancyhdr,dutchcal,appendix}
\usepackage{xcolor}

\newtheorem*{example}{\bf Example}

\newtheorem{theorem}{\bf Theorem}
\newtheorem{proposition}[theorem]{\bf Proposition}
\newtheorem{definition}[theorem]{\bf Definition}
\newtheorem{Theorem}{\bf Theorem}

\newtheorem{lemma}[theorem]{\bf Lemma}

\def\C{{\mathbb C}}
\def\N{{\mathbb N}}

\def\Q{{\mathbb Q}}

\def\P{\mathbb{P}}

\title{Complex dynamics of birational maps of $\mathbb{P}^k$ defined over a number field}
\author{Thomas Gauthier}
\address{Laboratoire de Math\'ematiques d'Orsay, B\^atiment 307, Universit\'e Paris-Saclay, 91405 Orsay Cedex, France}
\email{thomas.gauthier1@universite-paris-saclay.fr}
\author{Gabriel Vigny}
\address{LAMFA, Universit\'e de Picardie Jules Verne, 33 rue Saint-Leu, 80039 AMIENS Cedex 1, FRANCE}
\email{gabriel.vigny@u-picardie.fr}
\thanks{The first authors is partially supported by the Institut Universitaire de France.}
\thanks{The second author is partially supported by the ANR QuaSiDy, grant n${}^\mathrm{o}$ ANR-21-CE40-0016}
\thanks{Mathematics~Subject~Classification~(2020):  	37P05, 37P30, 37F80.}

\begin{document}

\maketitle
%\noindent \textbf{Keywords. HÃ©non maps, canonical height, algebraic family of rational maps, arithmetic characterizations of stability.} 
%
%\medskip
%
%\noindent \textbf{Mathematics~Subject~Classification~(2010):  	37P15, 37P30, 11G50, 37P35, 37F45.}

\begin{abstract}
	Jonsson and Reschke \cite{Jonsson-Reschke} showed that birational selfmaps on projective surface defined over a number field satisfy the energy condition of Bedford and Diller \cite{Bedford-Diller} so their ergodic properties are very well understood. Under suitable hypotheses on the indeterminacy loci, we extend that result to birational maps $\P^k \dashrightarrow \P^k$, $k\geq 2$, defined over a number field, showing that they satisfy a similar energy condition introduced by De Thélin and the second author \cite{ThelinVigny1}. As a consequence, we can construct for such maps their Green measure and deduce several important ergodic consequences. 
	
	Under a mild additional hypothesis, we show that generic sequences of Galois invariant subset of periodic points equidistribute toward the Green measure.  
\end{abstract}

\section{Introduction}

Let $f:X\dashrightarrow X$ be a birational map defined on a complex projective surface. Assume that the action of $f$ on the cohomology $H^{1,1}(X)$ has  spectral radius $\lambda_1(f)>1$, which is necessary to have positive entropy. In \cite{Bedford-Diller}, Bedford and Diller introduced an \emph{energy condition} which can be restated as
\[ \sum_{n\in \N } \frac{1}{\lambda_1(f)^n} \log d(f^{-n}(I_f), I_{f^{-1}}) >- \infty \ \mathrm{and} \ \sum_{n\in \N } \frac{1}{\lambda_1(f)^n} \log d(f^n(I_{f^{-1}}), I_{f}) ) >- \infty, \]
where $I_f$ (resp. $I_{f^{-1}}$) is the indeterminacy locus of $f$ (resp. of $f^{-1}$).
Under that condition, the authors  managed to construct a natural mixing hyperbolic measure for $f$ which does not charge curves, and combining with results of Dujardin \cite{Dujardin_laminar}, we have that this measure is of maximal entropy $\log \lambda_1(f)$ and describes the equidistribution of saddle periodic points. In other words, the energy condition is a natural condition to extend the ergodic properties of Hénon maps to birational maps on a complex projective surface (e.g. \cite{BedfordSmillie1, BedfordLyubichSmillie}). 

It is then natural to try and produce examples of birational maps satisfying the energy condition. Our source of inspiration in this article is the work of Jonsson and Reschke \cite{Jonsson-Reschke} which says the following: assume $f:X\dashrightarrow X$ is a birational maps defined on a complex projective surface $X$ where $X$ and $f$ are defined over a number field, assume that $\lambda_1(f)>1$, then, up to birational change of coordinates
\[\xymatrix{X'\ar[d]_{\pi} \ar@{-->}[r]^{f'}& X'\ar[d]_{\pi} \\
	X\ar@{-->}[r]_{f}  &   X }
\]
the map $f'$ satisfies the energy condition, also, as the measure $\mu_{f'}$ does not charge curve, it descends to a measure $\mu_f$ with the same ergodic properties. Note that the choice of the birational model is to ensure that the map $f'$ is \emph{algebraically stable} (see below), and when $f : \P^2  \dashrightarrow \P^2$ already is algebraically stable, then there is no need to change the model. Roughly speaking, combining \cite{Jonsson-Reschke} and \cite{Bedford-Diller} tells us that we have a very good understanding of birational maps defined over a number field.\\

In the present note, we focus on higher dimension and we consider the dynamics of birational maps $f:\mathbb{P}^k\dashrightarrow \mathbb{P}^k$ defined over a number field. Just like birational maps of $\P^2(\C)$ generalize Hénon maps, the maps we consider are natural generalizations of Hénon-Sibony maps, which are the polynomials automorphisms of $\C^k$ for which  $I_{f^{-1}} \cap I_{f}=\varnothing$ and whose dynamics is very well understood, see~\cite{Sibony}. In other words, we make the assumption that
there is an integer $1\leq s\leq k-1$ such that 
\begin{align}
	\dim I_f=k-s-1 \quad \text{and} \quad \dim I_{f^{-1}}=s-1,\tag{$\dag$}\label{gooddim}
\end{align}
where $I_f$ (resp. $I_{f^{-1}}$) is the indeterminacy locus of $f$ (resp. of $f^{-1}$).
We say that $f$ satisfies the \emph{improved algebraic stability condition} if we have
\begin{align}
\bigcup_{n\geq0}f^n(I_{f^{-1}})\cap I_f=\bigcup_{n\geq0}f^{-n}(I_{f})\cap I_{f^{-1}}=\varnothing.\tag{$\star$}\label{AS}
\end{align}
 Let $d$ (resp. $\delta$) be the algebraic degree of $f$ (resp. $f^{-1}$). Under the hypothesis \eqref{gooddim} and \eqref{AS}, one can show that $d^s=\delta^{k-s}$ and that the algebraic degree of $f^n$ (resp. $f^{-n}$) is $d^n$ (resp. $\delta^n$), see~\cite{ThelinVigny1}. De Thélin and the second author introduced for such maps a finite energy condition which generalizes Bedford-Diller's condition~\cite[Definition 3.1.9]{ThelinVigny1}, which can be stated as
\begin{align}\label{energy_higher2}
\left\{\begin{array}{l}
\displaystyle \sum^{\infty}_{n=0} \frac{1}{d^{sn}} \int_{f^n(I_{f^{-1}})} \phi \cdot f^*(\omega^{s-1}) >-\infty  \  \text{and} \\
\displaystyle  \sum^{\infty}_{n=0} \frac{1}{\delta^{n(k-s)}} \int_{f^{-n}(I_f)} \psi \cdot (f^{-1})^*(\omega^{k-s-1}) >-\infty. 
 \end{array}	\right.
\end{align}
where $\omega$ is the Fubini-Study form on $\P^k(\C)$ and where $\phi$ (resp. $\psi$) is a quasi-potential of $d^{-1} f^*(\omega)$, i.e. $d^{-1} f^*(\omega)=\omega + dd^c \phi$ (resp. a quasi-potential of $\delta^{-1} (f^{-1})^*(\omega)$, i.e. $\delta^{-1} (f^{-1})^*(\omega)=\omega + dd^c \psi$). Under that condition, we can construct a mixing and hyperbolic measure of maximal entropy $s \log d$, see~\cite{ThelinVigny1}. The main result of the paper is the following.

\begin{Theorem}\label{tm:good-dynamics}
Any birational map $f:\mathbb{P}^k\dashrightarrow \mathbb{P}^k$ which satisfies \eqref{AS} and \eqref{gooddim} and which is defined over a number field satisfies the finite energy condition \eqref{energy_higher2}. 

In particular, its Green measure $\mu_f$ is well-defined, mixing, hyperbolic, and has maximal entropy $s\log d$.
\end{Theorem}
As in \cite{Jonsson-Reschke}, the idea is to interpret \eqref{energy_higher2} as the local contribution of a canonical height of the subvariety $I_f$ at an archimedean place of the field of definition of $f$. The difficulty is that $I_f$ is not a finite set of points anymore so we have to use the notion of height of subvarieties following Zhang \cite{Zhang-positivity} and use Chambert-Loir's interpretation \cite{ACL} of the arithmetic intersection to relate local contributions of the height with \eqref{energy_higher2}.

In \cite{Jonsson-Reschke}, algebraic stability is not assumed. In our case, we start with that assumption for two reasons: first, \eqref{energy_higher2} was only defined for maps on $\P^k(\C)$ where the super-potentials theory of Dinh and Sibony \cite{DinhSibonysuper} is efficient, second, it is not clear that one can even find a suitable model where a birational map is algebraically stable in dimension $>2$ (Diller and Favre made it possible in dimension $2$ \cite{DillerFavre} but the recent work \cite{BDJK} of Bell, Diller, Jonsson, and Krieger shows it is not true anymore for $k=3$). Nevertheless, we show in $§$ \ref{examples} how to produce many examples which are not Hénon-Sibony maps.

\bigskip

Let us now move to the problem of the equidistribution of periodic points toward $\mu_f$. For Hénon-Sibony maps, the equidistribution is due to Dinh and Sibony \cite{DinhSibonyequi} using complex methods and to Lee \cite{Lee-Henon} in the arithmetic setting. We want to extend Lee's result to maps satisfying \eqref{AS}, \eqref{gooddim}, and \eqref{energy_higher2}. For that, we rely on De Th\'elin and Nguyen Van Sang~\cite{Thelin-Nguyen} who proved that for such maps, isolated periodic points are Zariski dense in $\mathbb{P}^k$. We then apply an arithmetic equidistribution theorem of the first author~\cite[Theorem~6]{Good-height} to deduce the following from Theorem~\ref{tm:good-dynamics}. 
\begin{Theorem}\label{tm:distrib-periodic}
Let $f:\mathbb{P}^k\dashrightarrow\mathbb{P}^k$ be a birational map defined over a number field and satisfying \eqref{AS} and \eqref{gooddim}. Assume in addition there is a constant $C\geq0$ such that
\begin{equation}\label{Lee}
	\frac{1}{d}h\circ f+\frac{1}{\delta}h\circ f^{-1}\geq \left(1+\frac{1}{d\delta}\right)h-C\end{equation}
on $(\mathbb{P}^k\setminus (I_f\cup I_{f^{-1}}))(\bar{\mathbb{Q}})$, where $h$ is the naive logarithmic height.
Then, there exist generic sequence $(F_i)_i$ of finite Galois invariant subsets of $\mathbb{P}^k(\bar{\mathbb{Q}})$ of periodic points of $f$ and, for such sequence, we have
\[\frac{1}{\# F_i}\sum_{x\in F_i}\delta_x\rightarrow \mu_f,\]
in the weak sense of probability measures on $\mathbb{P}^k(\mathbb{C})$.
\end{Theorem}
We construct the canonical height under \eqref{AS} and \eqref{gooddim} (Proposition~\ref{prop:canonical-height}), but we need the additional assumption \eqref{Lee} to derive the equidistribution theorem~\ref{tm:distrib-periodic} (even though we expect this hypothesis is not necessary). The condition \eqref{Lee} is not true in general but by Kawaguchi~\cite{Kawaguchi_inequality} (or Lee \cite{Lee_inequality}), it is true for Hénon-Sibony maps and thus for maps of the form $A \circ f$ where $A$ is an automorphism  and $f$ is a Hénon-Sibony map both defined over a number field, see~\S\ref{examples} for examples.  

\smallskip

It is worth noticing that apart from the case of Hénon-Sibony maps (see~\cite{Lee-Henon} and~\cite{DinhSibonyequi}), this is the first equidistribution result for birational maps in dimension at least 3. Remark also that the canonical height is in general \emph{not} a height associated to a $\mathbb{R}$-divisor or a $b$-divisor, since the set $\bigcup_{n\geq1}f^{n}(I_{f^{-1}})$ can be Zariski dense in $\mathbb{P}^k$ (see~\S~4.3). In particular, to our knowledge Theorem~6 from \cite{Good-height} is the only known theorem of equidistribution of (quasi-)small points that can be applied here.

\subsection*{Acknowledgments}
We would like to thank Chong Gyu Lee for explanations on condition \eqref{Lee} and the fact that it is not true in general. We also thank Nguyen Bac Dang for many discussions at the early stage of the article.

\section{Preliminaries}
\subsection{Metrized line bundles and mutual energy}
Let $(K,|\cdot|)$ be an algebraically closed field of characteristic zero which is complete with respect to a non-trivial absolute value. Let $X$ be a smooth projective variety of dimension $q\geq2$ and let $L$ be an ample line bundle on $X$, both defined over $K$. In what follows, we denote by $X^{\mathrm{an}}$ the Berkovich analytification of $X$.

\medskip

We let $\|\cdot\|_1$ and $\|\cdot\|_2$ be two semi-positive continuous metrics on $L$ and denote by $c_1(L,\|\cdot\|_i)$ the curvature form associated with $\|\cdot\|_i$. The continuous function
\[\phi:=\log\frac{\|\cdot\|_1}{\|\cdot\|_2}:X^{\mathrm{an}}\to\mathbb{R}\] 
defines a continuous metric on $\mathscr{O}_X$, which is a model metric if both $\|\cdot\|_1$ and $\|\cdot\|_2$ are model metrics.
The \emph{mutual energy} of $\|\cdot\|_1$ and $\|\cdot\|_2$ on $X$ is defined as
\[\mathrm{E}_X(L,\|\cdot\|_1,\|\cdot\|_2):=\frac{1}{(q+1)}\sum_{j=0}^q\int_{X^{\mathrm{an}}}\phi \cdot c_1(L,\|\cdot\|_1)^j\wedge c_1(L,\|\cdot\|_2)^{q-j}.\]

\subsection{Adelic metrics and arithmetic intersection}
Let $X$ be a projective variety of dimension $k$, and let $L_0,\ldots,L_k$ be $\mathbb{Q}$-line bundles on $X$, all defined over a number field $\mathbb{K}$.
Assume $L_i$ is equipped with an adelic continuous metric $\{\|\cdot\|_{v,i}\}_{v\in M_\mathbb{K}}$ and denote $\bar{L}_i:=(L_i,\{\|\cdot\|_v\}_{v\in M_\mathbb{K}})$. Assume $\bar{L}_i$ is  semi-positive for $1\leq i\leq k$. 
 Fix a place $v\in M_\mathbb{K}$. Denote by $X_v^\mathrm{an}$ the Berkovich analytification of $X$ at the place $v$. We also let $c_1(\bar{L}_i)_v$ be the curvature form of the metric $\|\cdot\|_{v,i}$ on $L_v^\mathrm{an}$.

Note that the hypothesis that $\bar{L}_i$ is adelic means in particular that for all but finitely many $v\in M_\mathbb{K}$, the metric $\|\cdot\|_{v,i}$ is a model metric on $L_{i,v}^{\mathrm{an}}$.

\medskip

In the sequel, for a given place $v\in M_\mathbb{K}$, denote by $\mathbb{C}_v$ an algebraically closed and complete extension of $(\mathbb{K},|\cdot|_v)$.

\medskip

For any closed subvariety $Y$ of dimension $q$ of $X$, the arithmetic intersection number $\left(\bar{L}_0\cdots\bar{L}_q|Y\right)$ is symmetric and multilinear with respect to the $L_i$'s. As observed by Chambert-Loir~\cite{ACL}, we can define $\left(\bar{L}_0\cdots\bar{L}_q|Y\right)$ inductively by
\[\left(\bar{L}_0\cdots\bar{L}_q|Y\right)=\left(\bar{L}_1\cdots\bar{L}_q|\mathrm{div}(s)\cap Y\right)+\sum_{v\in M_\mathbb{K}}n_v\int_{Y_v^{\mathrm{an}}}\log\|s\|^{-1}_v \bigwedge_{j=1}^qc_1(\bar{L}_i)_v,\]
for any global section $s\in H^0(X,L_0)$ such that the intersection $\mathrm{div}(s)\cap Y$ is proper. In particular, if $L_0$ is the trivial bundle and $\|\cdot\|_{v,0}$ is the trivial metric at all places but a finite set $S$ of places of $\mathbb{K}$, this gives
\begin{align}
\left(\bar{L}_0\cdots\bar{L}_q|Y\right)=\sum_{v\in S}n_{v}\int_{Y_{v}^{\mathrm{an}}}\log\|s\|^{-1}_{v,0} \bigwedge_{j=1}^qc_1(\bar{L}_i)_{v}.\label{eq:adelic}
\end{align}
When $\bar{L}$ is an ample $\mathbb{Q}$-line bundle endowed with a semi-positive continuous adelic metric, following Zhang~\cite{Zhang-positivity}, we can define $h_{\bar{L}}(Y)$ as
\[ h_{\bar{L}}(Y):=\frac{\left(\bar{L}^{q+1}|{Y}\right)}{(q+1)[\mathbb{K}:\mathbb{Q}]\deg_{L}(Y)},\]
where $\deg_L(Y)=(L_{|Y})^q$ is the volume of the line bundle $L$ restricted to $Y$.

\subsection{Arithmetic intersection and mutual energies}
Pick a smooth projective variety $X$ and an ample line bundle $L$, both defined over a number field $\mathbb{K}$. Let $\{\|\cdot\|_{v,1}\}_{v\in M_\mathbb{K}}$ and $\{\|\cdot\|_{v,2}\}_{v\in M_\mathbb{K}}$ be two adelic semi-positive metrics on $L$ and denote by $\bar{L}_i:=(L,\{\|\cdot\|_{v,i}\}_{v\in M_\mathbb{K}})$. For a given place $v\in M_\mathbb{K}$,  we denote by $\mathrm{E}_{X,v}(L,\|\cdot\|_{v,1},\|\cdot\|_{v,2})$ the mutual energy of the metrics $\|\cdot\|_{v,1}$ and $\|\cdot\|_{v,2}$ on $L_v^\mathrm{an}$.

\begin{lemma}\label{lm:energy}
With the above notations:
\[\frac{1}{q+1}\left((\bar{L}_1^{q+1}|X)-(\bar{L}_2^{q+1}|X)\right)=\sum_{v\in M_\mathbb{K}}n_v\cdot\mathrm{E}_{X,v}(L,\|\cdot\|_{v,1},\|\cdot\|_{v,2}).\]
\end{lemma}
\begin{proof}
Since the arithmetic intersection product is multilinear, we  have
\begin{align*}
(\bar{L}_1^{q+1}|X)-(\bar{L}_2^{q+1}|X) & = \sum_{j=0}^q\left((\bar{L}_1-\bar{L}_2)\cdot \bar{L}_1^j\cdot\bar{L}_2^{q-j}|Y\right).
\end{align*}
 Note that by assumption, $\bar{L}_1-\bar{L}_2$ is the trivial bundle endowed with an adelic continuous metric which vanishes for all $v\notin S$ ($S\subset M_\mathbb{K}$ is a finite set such that for any $v\notin S$, we have $\|\cdot\|_{v,2}=\|\cdot\|_{v,1}$ as metrics on $L_v^\mathrm{an}$). Using \eqref{eq:adelic} with the constant section $s\equiv1$ of the trivial bundle, we find
\begin{align*}
\left((\bar{L}_1-\bar{L}_2)\cdot \bar{L}_1^j\cdot\bar{L}_2^{q-j}|Y\right) & = \sum_{v\in S}n_v\int_{X_v^\mathrm{an}}\phi_v\cdot c_1(\bar{L}_1)_v^j\wedge  c_1(\bar{L}_2)_v^{q-j},
\end{align*}
where $\phi_v=\log(\|\cdot\|_{1,v}/\|\cdot\|_{2,v})$. We thus get
\[(\bar{L}_1^{q+1}|X)-(\bar{L}_2^{q+1}|X)=(q+1)\sum_{v\in S}n_v\cdot \mathrm{E}_{X,v}(L,\|\cdot\|_{v,1},\|\cdot\|_{v,2}).\]
Together with the fact that $\mathrm{E}_{X,v}(L,\|\cdot\|_{v,1},\|\cdot\|_{v,2})=0$ for all $v\notin S$, this gives the lemma.
\end{proof}

%By definition, we have
%\begin{align*}
%h_{\bar{L}_1}(X)-h_{\bar{L}_2}(X)=\frac{1}{(q+1)\deg_L(X)[\mathbb{K}:\mathbb{Q}]}\left((\bar{L}_1^{q+1}|X)-(\bar{L}_2^{q+1}|X)\right)
%\end{align*}

\subsection{Dynamical degrees and algebraic stability}
Let $K$ be an algebraically closed field of characteristic zero and let $f:\mathbb{P}^k\dashrightarrow \mathbb{P}^k$ be a dominant rational map defined over $K$. Recall that the $j$-\emph{dynamical degree} of $f$ can be computed as

\[\lambda_j(f):=\lim_{n\to\infty}\left((f^n)^*(L^j)\cdot L^{k-j}\right)^{1/n},\]
for any ample line bundle $L$ on $\mathbb{P}^k_K$ (see \cite{RS, DinhSibonydegree} for the complex case and \cite{Truong-dyn-degree,Bac-dyn-degree} for arbitrary characteristic).

\medskip

Assume furthermore $f:\mathbb{P}^k\dashrightarrow \mathbb{P}^k$ is a birational map which satisfies the improved algebraic stability condition \eqref{AS} and the dimension hypothesis \eqref{gooddim}. Let $d$ be the algebraic degree of $f$ and $\delta$ be the algebraic degree of $f^{-1}$. 
In this case, by \cite[$§$~3.1]{ThelinVigny1}, we have $\lambda_1(f)=d$ and $\lambda_1(f^{-1})=\delta$,
with $\lambda_j(f)=\lambda_1(f)^j=d^j$ for all $1\leq j\leq s$ and with $\lambda_\ell(f)=\lambda_1(f^{-1})^{k-\ell}=\delta^{k-\ell}$ for all $s\leq \ell\leq k$. In particular, $d^s=\delta^{k-s}$.

\subsection{The finite energy condition over a metrized field}\label{sec:DTV}
Let $f:\mathbb{P}^k\dashrightarrow \mathbb{P}^k$ be a dominant rational map defined over an algebraically closed field $K$ of characteristic zero.  Let $d=$ be the algebraic degree of $f$. Let $1\leq s\leq k-1$ be such that $\dim I_f=k-s-1$ where $I_f$ is the indeterminacy locus of $f$. Assume that for all $q\leq s$,  \[(\lambda_{q}(f))^n=\left((f^n)^*(\mathscr{O}_{\mathbb{P}^k}(1)^{q})\cdot \mathscr{O}_{\mathbb{P}^k}(1)^{k-q}\right)= d^{n q}.\]
\begin{definition}
Let $X\subsetneq \mathbb{P}^k$ be a closed subvariety. We say $X$ is $f$-\emph{good} if 
\[\bigcup_{n\geq0}f^n(X)\cap I_f=\varnothing.\]
\end{definition} 
Observe that if $X$ is $f$-good, then necessarily, $\dim X + \dim I_f \leq k-1$ so $\dim X \leq s$. 
Assume in addition that $(K,|\cdot|)$ is complete. Let $F:\mathbb{A}^{k+1}\to\mathbb{A}^{k+1}$ be a polynomial lift of $f$ defined over $K$, i.e. $F=(F_0,\ldots,F_k)$ with $F_i\in K[X_0,\ldots,X_k]$ homogeneous of degree $d$ with $\pi \circ F= f\circ \pi$ and $I_f=\pi(F^{-1}\{0\})$, where $\pi:\mathbb{A}^{k+1}\setminus\{0\}\to\mathbb{P}^k$ is the canonical projection. We define
\[\varphi_f(x):=\frac{1}{d}\log\|F(p)\|-\log\|p\|,\]\\
for all $x\in (\mathbb{P}^k\setminus I_f)(K)$ and $p\in\mathbb{A}^{k+1}(K)\setminus\{0\}$ with $\pi(p)=x$. If we equip $\mathscr{O}_{\mathbb{P}^k}(1)$ with a model metric
$\|\cdot\|_0$, we define a (singular) metric on $\mathscr{O}_{\mathbb{P}^k}(1)$ by letting $\|\cdot\|_f:=\|\cdot\|_0e^{-\varphi_f}$.
The singularities of the metric $\|\cdot\|_f$ are contained in $I_f$ and, for any closed subvariety $X\subsetneq \mathbb{P}^k$ with $X\cap I_f=\varnothing$, the line bundle $(\frac{1}{d}f^*\mathscr{O}_{\mathbb{P}^k}(1)-\mathscr{O}_{\mathbb{P}^k}(1))|_X$ is nothing but the trivial bundle $\mathscr{O}_X$ on $X$ and $\|\cdot\|_f$ is a model metric on $L:=\mathscr{O}_{\mathbb{P}^k}(1)|_X$.

\begin{definition}\label{def_metrization}
Let $X\subsetneq \mathbb{P}^k$ be an $f$-good closed subvariety of dimension $q$ and let $L_n:=\mathscr{O}_{\mathbb{P}^k}(1)|_{f^n(X)}$ and $\|\cdot\|_{f,n}$ and $\|\cdot\|_n$ be the metrizations on $L_n$ induced by $\|\cdot\|_f$ and $\|\cdot\|$ respectively. We say $(f,X)$ satisfies the \emph{finite energy condition} $(\mathrm{E})$ if 
\begin{align}
\sum_{n=0}^\infty \frac{1}{d^{n(q+1)}}\mathrm{E}_{f^n(X)}(L_n,\|\cdot\|_{f,n},\|\cdot\|_n)>-\infty.\tag{$\mathrm{E}$}
\end{align}
\end{definition}

\section{Finite energy condition for maps defined over a number  field}
%Let $\mathbf{k}$ be either $\mathbb{Q}$ or the function field $\mathbb{C}(S)$, where $S$ is a smooth projective variety defined over $\mathbb{C}$. Let $\mathbb{K}$ be a finite extension of $\mathbf{k}$. In both cases, $\mathbb{K}$ is a product formula field and for any $v\in M_\mathbb{K}$, if we let $n_v:=[\mathbb{K}_v:\mathbf{k}_v]$, then 
%\[\prod_{v\in M_\mathbb{K}}|x|_v^{n_v}=1, \ \text{for all} \ x\in \mathbb{K}^\times.\]
%If $\mathbf{k}$ is a function field, we can choose $|\cdot|_v$ so that $n_v=1$ for all $v\in M_\mathbb{K}$.
%
%Assume now $f:\mathbb{P}^k\dashrightarrow \mathbb{P}^k$ is a dominant rational map, defined over $\mathbb{K}$. Up to replacing $\mathbb{K}$ by a finite extension, we can assume $I_f$ is also defined over $\mathbb{K}$. 
%%We assume there is an ample line bundles $L$ on $X$, also defined over $\mathbb{K}$, such that $f^*L=dL$, where $d=\lambda_1(f)$. Let $N\geq1$ be an integer such that $N L$ is very ample. For some $R\geq1$, this induces an embedding $\iota:X\hookrightarrow \mathbb{P}^{R}$ for which $N L=\iota^*\mathscr{O}_{\mathbb{P}^{R}}(1)$.
%
%We endow $\mathscr{O}_{\mathbb{P}^{k}}(1)$ with the naive adelic metrization $\bar{\mathscr{O}}_{\mathbb{P}^{k}}(1)$.
%Note that for any closed subvariety $X\subset \mathbb{P}^k$ defined over $\mathbb{K}$, we have $h_{\mathrm{nv}}(X)\geq0$, see, e.g., \cite{Zhang-positivity,Gubler}.

\subsection{A general result for dominant rational maps}
Let now $f:\mathbb{P}^k\dashrightarrow \mathbb{P}^k$ be a dominant rational map of degree $d$ defined over a number field $\mathbb{K}$. Let $1\leq s\leq k-1$ be such that $\dim I_f=k-s-1$ where $I_f$ is the indeterminacy locus of $f$. Assume that for all $q\leq s$,  \[(\lambda_{q}(f))^n=\left((f^n)^*(\mathscr{O}_{\mathbb{P}^k}(1)^{q})\cdot \mathscr{O}_{\mathbb{P}^k}(1)^{k-q}\right)= d^{n q}.\]
For any place $v\in M_\mathbb{K}$, we let
$\|\cdot\|_{f,v}$ and $\varphi_{f,v}$ be defined as in $§$~\ref{sec:DTV}. The metric $\|\cdot\|_{f,v}$ induces a singular metric on $\mathscr{O}_{\mathbb{P}^k}(1)^\mathrm{an}_v$ with singular locus exactly $I_{f,v}^{\mathrm{an}}$ and the function $\varphi_{f,v}$ extends as a continuous function on $\mathbb{P}^{k,\mathrm{an}}_{v}\setminus I_{f,v}^{\mathrm{an}}$.

The main result of this section is the following version of \cite[Theorem 5.1]{Jonsson-Reschke} to our case, in which we use the notations of Definition~\ref{def_metrization} and add a $v$ to precise the dependence on the choice of the place $v \in M_\mathbb{K}$. 

\begin{theorem}\label{tm:DTV-global}
Let $X\subsetneq \mathbb{P}^k$ be an $f$-good subvariety defined over $\mathbb{K}$. There is a constant $C\geq0$ such that for any place $v\in M_\mathbb{K}$, we have
\begin{align*}
\sum_{n=0}^\infty \frac{1}{d^{n(q+1)}}\mathrm{E}_{f^n(X)}(L_n,\|\cdot\|_{f,n,v},\|\cdot\|_{n,v})\geq -\frac{[\mathbb{K}:\mathbf{k}]\deg(X)}{n_{v}}\left(h_{\mathrm{nv}}(X)+C\right).
\end{align*}
In particular, the pair $(X,f)$ satisfies the finite energy condition $(\mathrm{E})$ over $\mathbb{C}_v$. 
\end{theorem}

\begin{proof}
As $X$ is $f$-good, we must have $q:=\dim X \leq s$. In particular, the degree of $(f^n)(X)$, counted with multiplicity, can be computed  by 
 \begin{equation}\label{eq_degree_iterate}
 \left( f^n(X)\cdot \mathscr{O}_{\mathbb{P}^k}(1)^{q} \right)= \left( X\cdot (f^n)^*(\mathscr{O}_{\mathbb{P}^k}(1)^{q}) \right) = d^{n q} \deg(X).
\end{equation}
 Fix an integer $n\geq0$ and recall that $L_{n}:=\mathscr{O}_{\mathbb{P}^k}(1)|_{f^n(X)}$. Let $\bar{L}_n=(L_n,\{\|\cdot\|_{n,v}\}_{v\in M_\mathbb{K}})$ be the semi-positive adelic ample line bundle induced by $\bar{\mathscr{O}}_{\mathbb{P}^k}(1)$ on $f^n(X)$.

Note that, since $X$ is $f$-good, the collection $\{\|\cdot\|_{f,n,v}\}_{v\in M_\mathbb{K}}$ of singular metrics on $\mathscr{O}_{\mathbb{P}^k}(1)$ induces a model metric on $L_n$. We denote by $\bar{L}_{f,n}$ the induced adelically metrized line bundle. By construction, we have $f^*\bar{L}_{n+1}=d\bar{L}_{f,n}$ and
\begin{align*}
(\bar{L}_{f,n}^{q+1}|f^n(X)) & =d^{-(q+1)}((f^*\bar{L}_{n+1})^{q+1}|f^n(X))=d^{-(q+1)}(\bar{L}_{n+1}^{q+1}|f^{n+1}(X)).
\end{align*}

% Note that by the invariance property $f^*L=dL$, we have 
%\[\deg_L(f^n(X))=((f^n)^*c_1(L)^q\cdot\{X\})=d^{nq}(c_1(L)^q\cdot\{X\})=d^{nq}\deg_L(X).\]
%For any $n\geq0$, the above gives
%\begin{align*}
%h_{\bar{L}_f}(f^n(X)) & =\frac{1}{(q+1)\deg_L(f^n(X))}(\bar{L}_f^{q+1}|f^n(X))\\
%& =\frac{1}{(q+1)d^{nq}\deg_L(X)}(\bar{L}_f^{q+1}|f^n(X)) \\
%& = \frac{1}{(q+1)d^{nq}\deg_L(X)}d^{-(q+1)}(\bar{L}^{q+1}|f^{n+1}(X))
%\end{align*}
%where we used that, 
%Since $\deg_L(f^{n+1}(X))=d^{q}\deg_L(f^n(X))$, this gives in particular 
%\[h_{\bar{L}_f}(f^n(X))=\frac{d^{-(q+1)}(\bar{L}^{q+1}|f^{n+1}(X))}{(q+1)\deg_L(f^n(X))}=\frac{1}{d}h_{\bar{L}}(f^{n+1}(X)).\]
We thus find
\begin{align*}
I_N & :=\sum_{n=0}^N\frac{1}{(q+1)d^{n(q+1)}}\left((\bar{L}_{f,n}^{q+1}|f^n(X))-(\bar{L}_n^{q+1}|f^n(X))\right)\\
 & =\sum_{n=0}^N\frac{1}{(q+1)d^{n(q+1)}}\left(\frac{1}{d^{q+1}}(\bar{L}_{n+1}^{q+1}|f^{n+1}(X))-(\bar{L}_n^{q+1}|f^n(X))\right)\\
& = \frac{1}{(q+1)d^{(q+1)(N+1)}}(\bar{L}_{N+1}^{q+1}|f^{N+1}(X))-\frac{1}{q+1}(\bar{L}_0^{q+1}|X)\\
&\geq-[\mathbb{K}:\mathbb{Q}]\deg(X)h_{\bar{L}_0}(X),
\end{align*}
since $(\bar{L}_{N+1}^{q+1}|f^{N+1}(X))=(q+1)[\mathbb{K}:\mathbb{Q}]\deg(f^{N+1}(X))h_{\bar{L}_{N+1}}(X)\geq0$.
According to Lemma~\ref{lm:energy}, for any $n\geq0$,
\begin{align*}
\frac{1}{q+1}\left((\bar{L}_{f,n}^{q+1}|f^n(X))-(\bar{L}_n^{q+1}|f^n(X))\right)=\sum_{v\in M_\mathbb{K}}n_v\cdot \mathrm{E}_{f^n(X),v}(L_n,\|\cdot\|_{f,n,v},\|\cdot\|_{n,v}).
\end{align*}
For any $N\geq0$, this leads to
\begin{align}
\sum_{n=0}^N\sum_{v\in M_\mathbb{K}}\frac{n_v}{d^{n(q+1)}} \mathrm{E}_{f^n(X),v}(L_n,\|\cdot\|_{f,n,v},\|\cdot\|_{n,v})\geq -[\mathbb{K}:\mathbb{Q}]\deg(X)h_{\bar{L}_0}(X).\label{eq:decomposition}
\end{align}
We now recall that, by the (strong) triangular inequality, for any $v\in M_\mathbb{K}$, there is $C_v\geq0$ such that
\begin{align}
\varphi_{f,v}\leq C_v \quad \text{on} \ \ (\mathbb{P}^k\setminus I_f)(\mathbb{C}_v),\label{ineg:trivial}
\end{align}
with $C_v=0$ for all but finitely many $v\in M_\mathbb{K}$, see, e.g., \cite{Silverman,Jonsson-Reschke}. As a consequence, for all $n\geq0$ and all $v\in M_\mathbb{K}$
\[\mathrm{E}_{f^n(X),v}(L_n,\|\cdot\|_{f,n,v},\|\cdot\|_{n,v})\leq C_v \cdot \deg_{L_n}(f^n(X))= C_v \cdot d^{nq}\deg(X),\]
where we used \eqref{eq_degree_iterate}.
Set $S:=\{v\in M_\mathbb{K}\, : C_v\neq0\}\subset M_\mathbb{K}$. For all $v\notin S$, we get $\mathrm{E}_{f^n(X);v}(L,\|\cdot\|_{f,n,v},\|\cdot\|_{n,v})\leq 0$.
Finally, we pick a place $v_0\in M_\mathbb{K}$ and $n\geq0$. By \eqref{eq:decomposition} and \eqref{ineg:trivial}, we have
\begin{align*}
\sum_{n=0}^N\frac{n_{v_0}d^{-n(q+1)}}{[\mathbb{K}:\mathbb{Q}]\deg(X)}\mathrm{E}_{f^n(X),v_0}(L_n,\|\cdot\|_{f,n,v_0},\|\cdot\|_{n,v_0}) & \geq -h_{\bar{L}_0}(X)-\sum_{v\in M_{\mathbb{K}}\setminus\{v_0\}}\sum_{n=0}^N\frac{n_v}{d^n}C_v\\
&\geq -h_{\bar{L}_0}(X)-\frac{dC}{d-1},
\end{align*}
where $C:=\sum_{v\in M_\mathbb{K}}n_vC_v=\sum_{v\in S}n_vC_v$, since $C_v\geq0$ for all $v\in M_\mathbb{K}$ and $C_v=0$ for all $v\notin S$. Since $\bar{L}_0=\mathscr{O}_{\mathbb{P}^k}(1)|_X$ and since this metrization induces the naive height, this concludes the proof.
\end{proof}
\subsection{Finite energy and canonical heights for birational maps}
In this section, we let $f:\mathbb{P}^k\dashrightarrow\mathbb{P}^k$ be a birational map defined over a number field $\mathbb{K}$ satisfying the improved algebraic stability assumption \eqref{AS} and \eqref{gooddim}. 

\begin{proof}[Proof of Theorem~\ref{tm:good-dynamics}]
Fix a complex place $v\in M_\mathbb{K}$. 
As $f$ satisfies \eqref{AS}, $(I_{f^{-1}},f)$ is a good pair. In particular, we can apply Theorem~\ref{tm:DTV-global} over $\mathbb{C}_v$, which means
\begin{align*}
	\sum_{n=0}^\infty \frac{1}{d^{ns}}\mathrm{E}_{f^n(I_{f^{-1}})}(L_n,\|\cdot\|_{f,n},\|\cdot\|_n)>-\infty,
\end{align*} 
with
\[\mathrm{E}_{f^n(I_{f^{-1}})}(L_n,\|\cdot\|_{f,n},\|\cdot\|_n):=\frac{1}{s}\sum_{j=0}^{s-1}\int_{f^n(I_{f^{-1}})}\varphi \cdot c_1(L_n,\|\cdot\|_{f,n})^j\wedge c_1(L_n,\|\cdot\|_n)^{s-1-j},\]
where $L_{n}=\mathscr{O}_{\mathbb{P}^k}(1)|_{f^n(I_{f^{-1}})}$, $\|\cdot\|_n$ is the naive metric $\|\cdot\|_0$ on $L_{n}$, and $\|\cdot\|_{f,n}$ is the metric $\|\cdot\|_0e^{-\varphi_{f}}$ on $L_n$ with \[\varphi_f(x)=\frac{1}{d}\log\|F(p)\|-\log\|p\|\]
for some lift $F$ of $f$. The finiteness of the sum implies (and is in fact equivalent) to the finiteness of the sum for $j=s-1$, see~\cite[Proof of Theorem~3.2.1]{ThelinVigny1}, and $\varphi_f$ is indeed a quasi-potential of $d^{-1}f^*(\omega)$. Using that $c_1(L_n,\|\cdot\|_{f,n})$ is the restriction of $d^{-1} f^*(\omega)$ to $f^n(I_{f^{-1}})$ implies the first part of \eqref{energy_higher2}. Finally, working with $(I_f,f^{-1})$ implies $f$ satisfies \eqref{energy_higher2}.

The fact that its Green measure $\mu_f$ is well-defined, mixing, hyperbolic, and has maximal entropy $s\log d$ is an immediate consequence of \cite[Theorems~4~\&~5]{ThelinVigny1}.
\end{proof}
The set of points with a well defined grand orbit is
\[\mathbb{P}_f^k:=\mathbb{P}^k\setminus\left(\bigcup_{n\geq0}f^n(I_{f^{-1}})\cup f^{-n}(I_f)\right).\]
We prove here the following.

\begin{proposition}\label{prop:canonical-height}
Assume $f$ satisfies assumption \eqref{AS} and \eqref{gooddim}. There exist canonical height functions $\widehat{h}_f^+,\widehat{h}_f^-:\mathbb{P}^k_f(\bar{\mathbb{Q}})\to\mathbb{R}_+$ such that 
\begin{center}
 $\widehat{h}_f^+(f(x))=d\widehat{h}_f^+(x)$ and $\widehat{h}_f^-(f^{-1}(x))=\delta\widehat{h}_f^-(x)$ for all $x\in \mathbb{P}_f^k(\bar{\mathbb{Q}})$.
 \end{center}
 In particular, if $x\in\mathbb{P}^k(\bar{\mathbb{Q}})$ is periodic, then $\widehat{h}_f^+(x)=\widehat{h}_f^-(x)=0$.

Moreover, if we assume there exists a constant $C\in\mathbb{R}$ such that 
\begin{align}
\frac{1}{d} h(f(x))+ \frac{1}{\delta} h(f^{-1}(x)) \geq \left( 1+ \frac{1}{d\delta}\right) h(x)+C, \quad x\in \mathbb{P}^k(\bar{\mathbb{Q}})\setminus(I_f(\bar{\mathbb{Q}})\cup I_{f^{-1}}(\bar{\mathbb{Q}})),\label{eq:bonneineg}
\end{align}
then there is a sequence of positive numbers $(\epsilon_n)_n$ such that $\epsilon_n\to0$ as $n\to\infty$ and 
\begin{equation}\label{cequonveut?}
\frac{1}{d^n}h\circ f^n(x)+\frac{1}{\delta^n} h \circ f^{-n} (x)  \leq  \hat{h}^+_f(x)+ \hat{h}^-_f(x)+ \epsilon_n, \quad x\in\mathbb{P}_f^k(\bar{\mathbb{Q}}).
\end{equation}
\end{proposition}

%We now let $x\in \mathbb{P}^k(\mathbb{C})$ be an isolated periodic point of $f$, with period $p\geq1$. As it is isolated, we have $x\in \mathbb{P}^k(\bar{\mathbb{Q}})$ and $f^n$ and $f^{-n}$ are well-defined at $x$ for any $n\geq1$, so that $x$ avoids the exceptional set of $\psi_n$ for any $n$. 

%In particular, taking the sum over $0\leq n\leq N-1$, this gives
%\[\frac{1}{d^{N}}h_{\mathrm{nv}}(f^{N}(x))+\frac{1}{\delta^N}h_{\mathrm{nv}}(f^{-N}(x))\leq \frac{dC_1}{d-1}+2h_{\mathrm{nv}}(x).\]
%dd

%~
%
%XXXXXX
%
%~
%
%
%XXXXXXX
%The exact same proof as that of \cite[Theorem~2.1]{Kawaguchi-Henon} gives the next result:
%
%\begin{lemma}
%Pick $n\geq1$ and a hyperplane $H\subset\mathbb{P}^k$. Then the divisor
%\[D_n:=\]
%is effective on $X_n$.
%\end{lemma}
%
%
%XXXXXXX

\begin{proof}
As before, we use that $\max\{\varphi_{f,v},\varphi_{f^{-1},v}\}\leq C_v$ on $\mathbb{P}_f^k(\bar{\mathbb{Q}})$ where $C_v=0$ for all but finitely many $v\in M_\mathbb{K}$. We deduce that 
\[\max\{\varphi_{f,v}\circ f^n,\varphi_{f^{-1},v}\circ f^{-n}\}\leq C_v \quad \text{on} \ \mathbb{P}_f^k(\bar{\mathbb{Q}}).\]
In particular, if $C_1$ is the constant $C_1:=\sum_{v\in M_\mathbb{K}}n_v\cdot C_v\in\mathbb{R}_+$, summing over all places and over all Galois conjugates of a point $x\in \mathbb{P}_f^k(\bar{\mathbb{Q}})$, we find
\begin{align}\label{ineg:good-height-ineg}
\left\{\begin{array}{ll}
\displaystyle\frac{1}{d^{n+1}}h_{\mathrm{nv}}(f^{n+1}(x))-\frac{1}{d^{n}}h_{\mathrm{nv}}(f^{n}(x))\leq \frac{1}{d^n}C_1,  & \text{and}\\
\displaystyle\frac{1}{\delta^{n+1}}h_{\mathrm{nv}}(f^{-n-1}(x))-\frac{1}{\delta^n}h_{\mathrm{nv}}(f^{-n}(x))\leq \frac{1}{\delta^n}C_1,&
\end{array}\right.
\end{align}
where $C_1$ is independent of $x\in \mathbb{P}_f^k(\bar{\mathbb{Q}})$ and of $n\geq0$. As in Kawaguchi \cite{Kawaguchi-Henon}, we deduce, following the arguments of \cite{CS-height}, that
the limits
\[\widehat{h}_f^+:=\limsup_{n\to\infty}\frac{1}{d^n}h_{\mathrm{nv}}\circ f^n \quad \text{and} \quad \widehat{h}_f^-:=\limsup_{n\to\infty}\frac{1}{\delta^n}h_{\mathrm{nv}}\circ f^{-n}\]
are well-defined functions $\widehat{h}_f^\pm:\mathbb{P}^k_f(\bar{\mathbb{Q}})\to\mathbb{R}_+$ and satisfy $\widehat{h}_f^+\circ f=d\widehat{h}_f^+$ and $\widehat{h}_f^-\circ f^{-1}=\delta \widehat{h}_f^-$. 

%If $x\in \mathbb{P}^k(\bar{\mathbb{Q}})$ is a periodic point of period $p\geq1$,  then $x\in\mathbb{P}_f^k(\bar{\mathbb{Q}})$ and, by construction, 
%\[\widehat{h}_f^+(x)=\widehat{h}_f^+(f^p(x))=d^p\widehat{h}_f^+(x) \quad \text{and} \quad \widehat{h}_f^-(x)=\widehat{h}_f^-(f^{-p}(x))=\delta^p\widehat{h}_f^-(x),\]
%so that $\widehat{h}_f^+(x)=\widehat{h}_f^-(x)=0$. 

\medskip

We now assume \eqref{eq:bonneineg} holds. We again follow ideas of \cite{Kawaguchi-Henon}. Let $D:=d\delta$ and $h':= h + \kappa$ where $\kappa= \frac{-CD}{D+1-d-\delta}$ is a constant chosen so that \eqref{eq:bonneineg} rephrases as
\begin{align}
\frac{1}{d} h'(f(x))+ \frac{1}{\delta} h'(f^{-1}(x)) \geq \left( 1+ \frac{1}{D}\right) h'(x).\label{eq:Lee}
\end{align} 
For $n\in \N^*$ and $x\in \mathbb{P}_f^k$, let us denote by $h'_n(x)$ the quantity 
\[h'_n(x):=\frac{1}{d^n} h'(f^n(x))+ \frac{1}{\delta^n} h'(f^{-n}(x)).\]
 We write $h'_0=h'$ for $n=0$. We let $c_n:= (D^n+1)/D^n$ for $n\geq 1$ and $c_0=1$. 
We shall prove by induction on $n$ that 
\[ h'_n \geq \frac{c_n}{c_{n-1}} h'_{n-1}\]
whenever all these quantities are well defined. The step $n=1$ is  \eqref{eq:Lee}. Assume now the inequality holds for some $n$ and compose \eqref{eq:Lee} with $f^n$ and $f^{-n}$:
\begin{align*}
	\frac{1}{d^{n+1}} h'(f^{n+1}(x))+ \frac{1}{\delta d^n} h'(f^{n-1}(x)) \geq c_1 \frac{1}{d^{n}} h'(f^n(x)) \\
	\frac{1}{d \delta^{n}} h'(f^{-(n-1)}(x))+ \frac{1}{\delta^{n+1}} h'(f^{-(n+1)}(x)) \geq c_1 \frac{1}{\delta^{n}} h'(f^{-n}(x)). 
\end{align*}
Summing we recognize
\begin{align*}
	h'_{n+1}+ \frac{1}{D} h'_{n-1} \geq c_1 h'_n.
\end{align*}
Using the induction hypothesis gives $h'_{n+1}+ \frac{c_{n-1}}{Dc_n} h'_{n} \geq c_1 h'_n$ so $h'_{n+1} \geq \left(c_1-\frac{c_{n-1}}{Dc_n} \right) h'_n= \frac{c_{n+1}}{c_n} h'_n$ by a straightforward computation. So $h'_n \geq \frac{c_n}{c_{n-1}} h'_{n-1}$ holds for all $n$. Multiplying that inequality for $m \geq n+1$ and passing to the limit give
\[\limsup h'_m \geq c_n h'_{n}.\]
Recall that $\hat{h}^+_f= \limsup d^{-n} h\circ f^n $ and $\hat{h}^-_f= \limsup \delta^{-n} h\circ f^{-n} $ so, replacing $h'_n$ by $d^{-n}h\circ f^n+\delta^{-n} h \circ f^{-n} +\kappa(d^{-n}+\delta^{-n})$ implies \eqref{cequonveut?}
for points in $\mathbb{P}_f^k$. 
\end{proof}
\section{Distribution of generic periodic points of birational maps}

\subsection{Arithmetic equidistribution for quasi-heights}
In this section, we let $X$ be a projective variety of dimension $k$ defined over a number field $\mathbb{K}$ and we fix a place $v\in M_\mathbb{K}$. For any $n\geq0$, we let $\psi_n:X_n\to X$  be a birational morphism and we let $L_n$  be a big and nef $\mathbb{Q}$-line bundle on $X_n$ endowed with a semi-positive adelic continuous metrization $\bar{L}_n$. We assume that
\begin{enumerate}
\item the sequence $\mathrm{vol}(L_n)$ converges to  constant $\mathrm{V}>0$ and the sequence of probability measures $(\mathrm{vol}(L_n)^{-1}(\psi_n)_*c_1(\bar{L}_n)^k_v)_n$ converges weakly to a probability measure $\mu_v$ on $X_v^\mathrm{an}$,
\item For any ample line bundle $M_0$ on $X$ and any adelic semi-positive continuous metrization $\bar{M}_0$ on $M_0$, there is a constant $C\geq0$ such that 
\[\left(\psi_n^*(\bar{M}_0)\right)^j\cdot \left(\bar{L}_n\right)^{k+1-j}\leq C,\]
for any $2\leq j\leq k+1$ and any $n\geq0$.
\end{enumerate}
\begin{definition}
The data $(X,\mu_v,X_n,\bar{L}_n)$ is a \emph{quasi-height} on $X$ at the place $v$.
\end{definition}

A sequence $(F_i)_i$ of Galois-invariant finite subsets of $X(\bar{\mathbb{Q}})$ is \emph{quasi-small} if $\psi_n^{-1}\{F_i\}$ is a finite subset of $X_n(\bar{\mathbb{Q}})$ for any $n\geq0$ and any $i$ and if the sequence
\[\varepsilon_n(\{F_i\}_i):=\limsup_ih_{\bar{L}_n}(\psi_n^{-1}(F_i))-h_{\bar{L}_n}(X_n)\] satisfies $\limsup_{n\to\infty}\varepsilon_n(\{F_i\})\leq 0$.

\bigskip

The following is proved in \cite{Good-height}:

\begin{theorem}[Equidistribution of quasi-small points]\label{tm:equidistrib}
Let $X$ be a projective variety defined over a number field $\mathbb{K}$, let $v\in M_\mathbb{K}$ and let $(X,\mu_v,X_n,\bar{L}_n)$ be a quasi-height on $X$ at the place $v$. For any quasi-small sequence $(F_m)_m$ of Galois-invariant finite subsets of $X(\bar{\mathbb{Q}})$ such that for any hypersurface $H\subset V$ defined over $\mathbb{K}$, we have
\[\# (F_n\cap H)=o(\# F_n), \quad \text{as} \ n\to+\infty,\]
the probability measure $\mu_{F_m,v}$ on $X_v^\mathrm{an}$ which is  equidistributed on $F_m$ converges to $\mu_v$ in the weak sense of measures, i.e. for any continuous function with compact support $\varphi\in\mathscr{C}^0(X_v^\mathrm{an})$, we have
\[\lim_{m\to\infty}\frac{1}{\# F_m}\sum_{y\in F_m}\varphi(y)=\int_{X_v^\mathrm{an}}\varphi\,\mu_v.\]
\end{theorem}

\subsection{Dynamical quasi-heights for birational maps} 

We now prove Theorem~\ref{tm:distrib-periodic}, applying Theorem~\ref{tm:equidistrib} above. Let $f:\mathbb{P}^k\dashrightarrow \mathbb{P}^k$ be a birational selfmap of $\mathbb{P}^k$ defined over a number field $\mathbb{K}$ and satisfying \eqref{AS} and \eqref{Lee}. Let $d:=\deg(f)$ and $\delta:=\deg(f^{-1})$. Recall that $d^s=\delta^{k-s}$. 
We now choose an embedding $\mathbb{K}\hookrightarrow\mathbb{C}$, and let $f:\mathbb{P}^k(\mathbb{C})\dashrightarrow \mathbb{P}^k(\mathbb{C})$ be the induced complex birational selfmap. By Theorem~\ref{tm:good-dynamics}, $f$ satisfies the hypothesis of \cite[Theorems~4~\&~5]{ThelinVigny1}. More precisely, the Green currents $T_f^j$ and $T_{f^{-1}}^{\ell}$ are well-defined for $1\leq j\leq s$ and for $1\leq \ell\leq k-s$ and satisfy
\[\lim_{n\to\infty}\frac{1}{d^{nj}}(f^n)^*\omega^j=T_f^j \quad \text{and} \quad  \lim_{n\to\infty}\frac{1}{\delta^{n\ell}}(f^{-n})^*\omega^\ell=T_f^\ell.\]
By \cite[Theorem~3.2.8]{ThelinVigny1}, the above  convergence is in the Hartogs' sense (which means that the super-potentials are almost decreasing to the super-potentials of the limits \cite{DinhSibonysuper}.)
Moreover, the measure $\mu_f:=T_f^s\wedge T_{f^{-1}}^{k-s}$ is mixing (with an exponential speed by \cite{Vigny_decay}) hence ergodic, and of maximal entropy $s\log d>0$. Since the currents $T_f^s$ and $T_{f^{-1}}^{k-s}$ are wedgeable by \cite[Theorem~3.4.1]{ThelinVigny1}, continuity of the wedge product under Hartogs convergence for wedgeable currents (see again \cite[Proposition~4.2.6]{DinhSibonysuper}) implies
\begin{equation}\label{eq:Hartogs}
	\mu_f=\lim_{n\to\infty}\frac{1}{d^{ns}}(f^n)^*(\omega^s)\wedge \frac{1}{\delta^{n(k-s)}}(f^{-n})^*(\omega^{k-s}).
\end{equation}
The measure $\mu_f$ satisfies $\int \log d(x,I_f) d\mu_f >-\infty$ and is hyperbolic
\[\chi_1\geq\cdots \geq \chi_s>0>\chi_{s+1}\geq\cdots\geq \chi_k,\]
where $\chi_i$ is the $i$-th Lyapunov exponent of $\mu_f$.

\medskip

We define $X_n$ as a finite sequence of blowups $\psi_n:X_n\to\mathbb{P}^k$ of $\mathbb{P}^k$ such that the maps $f^n\circ\psi_n$ and $f^{-n}\circ \psi_n$ extend as morphisms $f_n,g_n:X_n\to\mathbb{P}^k$. This amounts to the fact that the following diagram commutes:
\begin{align*}
\xymatrix{ &X_n\ar[d]_{\psi_n} \ar[rd]^{f_n}\ar[ld]_{g_n}& \\
\mathbb{P}^k& \mathbb{P}^k\ar@{-->}[r]_{f^n} \ar@{-->}[l]^{f^{-n}}  & \mathbb{P}^k}
\end{align*}

We let $\bar{L}_0$ be the classical adelic metrization on $\mathscr{O}_{\mathbb{P}^k}(1)$, so that in particular $h_{\bar{L}_0}\geq0$ on $\mathbb{P}^k(\bar{\mathbb{Q}})$ and, if $\omega$ be the Fubini Study form on $\mathbb{P}^k$, then $\omega$ is the curvature form of $\bar{L}_0$ over $\mathbb{C}$. Let
\[\bar{L}_n:=\frac{1}{d^n}f_n^*\bar{L}_0+\frac{1}{\delta^n}g_n^*\bar{L}_0.\]
In what follows, we denote by $h$ the naive logarithmic height on $\mathbb{P}^k$. We prove here the following.

\begin{proposition}\label{prop:quasi-height}
Let $f$ be a birational selfmap of $\mathbb{P}^k$ satisfying \eqref{AS} and \eqref{gooddim}. Assume in addition there is a constant $C\geq0$ such that
\[\frac{1}{d}h\circ f+\frac{1}{\delta}h\circ f^{-1}\geq \left(1+\frac{1}{d\delta}\right)h-C\]
on $(\mathbb{P}^k\setminus (I_f\cup I_{f^{-1}}))(\bar{\mathbb{Q}})$. With the above notations, $(\mathbb{P}^k,\mu_f,X_n,\bar{L}_n)$ is a quasi-height at the complex place and any sequence $(F_i)$ of Galois-invariant finite sets $F_i\subset\mathbb{P}^k(\bar{\mathbb{Q}})$ of periodic points of $f$ is quasi-small.
\end{proposition}

\begin{proof}
We first check condition (1). Note that, since $f_n$ and $g_n$ are generically finite dominant morphisms, and since $\mathscr{O}_{\mathbb{P}^k}(1)$ is ample, the $\mathbb{Q}$-line bundle $L_n$ is big and nef. In particular, $\mathrm{vol}(L_n)=(L_n^k)$.
Then we can compute
\begin{align*}
\mathrm{vol}(L_n) &=\int_{X_n(\mathbb{C})}\left(\frac{1}{d^n}f_n^*\omega+\frac{1}{\delta^n}g_n^*\omega\right)^k\\
& = \sum_{j=0}^k{k \choose j} \frac{1}{(d^{j}\delta^{k-j})^n}\int_{X_n(\mathbb{C})}f_n^*(\omega^j)\wedge g_n^*(\omega^{k-j}).
\end{align*}
Fix $0\leq j\leq k$. As $\omega$ is a smooth form and as $f_n$ and $g_n$ are morphisms, for any closed subvariety $Y\subsetneq X_n$, the measure $f_n^*(\omega^j)\wedge g_n^*(\omega^{k-j})$ does not give mass to $Y(\mathbb{C})$. Let $Y_n:=\psi_n^{-1}\left(\bigcup_{0\leq \ell\leq n}f^\ell(I_{f^{-1}})\cup f^{-\ell}(I_{f})\right)$, so that $\psi_n$ is an isomorphism from $X_n\setminus Y_n$ to its image $Z_n:=\bigcup_{0\leq \ell\leq n}f^\ell(I_{f^{-1}})\cup f^{-\ell}(I_{f})$. We then have
\begin{align*}
\int_{X_n(\mathbb{C})}f_n^*(\omega^j)\wedge g_n^*(\omega^{k-j}) & =\int_{X_n(\mathbb{C})\setminus Y_n(\mathbb{C})}f_n^*(\omega^j)\wedge g_n^*(\omega^{k-j})\\
& = \int_{X_n(\mathbb{C})\setminus Y_n(\mathbb{C})}\psi_n^*\left((f^n)^*(\omega^j)\wedge (f^{-n})^*(\omega^{k-j})\right)\\
& = \int_{\mathbb{P}^k(\mathbb{C})\setminus Z_n(\mathbb{C})}(f^n)^*(\omega^j)\wedge (f^{-n})^*(\omega^{k-j}).
\end{align*}
We now use Bézout Theorem for currents to find
\begin{align*}
\int_{\mathbb{P}^k(\mathbb{C})\setminus Z_n(\mathbb{C})}(f^n)^*(\omega^j) & \wedge (f^{-n})^*(\omega^{k-j})\leq \\
&\left(\int_{\mathbb{P}^k(\mathbb{C})}(f^n)^*(\omega^j)\wedge \omega^{k-j}\right)\times \left(\int_{\mathbb{P}^k(\mathbb{C})}(f^{-n})^*(\omega^{k-j})\wedge \omega^{j}\right),
\end{align*}
with equality when $j=s$. By assumption, this gives
\begin{align*}
\int_{X_n(\mathbb{C})}f_n^*(\omega^j)\wedge g_n^*(\omega^{k-j}) \leq \lambda_j(f)^n\times \lambda_{k-j}(f^{-1})^n,
\end{align*}
and we have proved that $(d^{-j}\delta^{j-k})^nf_n^*(\omega^j)\wedge g_n^*(\omega^{k-j})$ has mass at most 
\[\left(\frac{\lambda_j(f)}{d^j}\frac{\lambda_{k-j}(f^{-1})}{\delta^{k-j}}\right)^n=\left\{\begin{array}{ll}
O\left({\delta}^{-n}\right) & \text{if} \ \ j<s,\\
1 & \text{if} \ \ j=s,\\
O\left({d}^{-n}\right) & \text{if} \ \ j>s.
\end{array}\right.\]
In particular, the volume of $L_n$ satisfies
\[\mathrm{vol}(L_n) \leq \sum_{j=0}^k{k \choose j} \left(\frac{\lambda_j(f)}{d^j}\frac{\lambda_{k-j}(f^{-1})}{\delta^{k-j}}\right)^n= {k \choose s} +o(1).\]
Moreover, for $j=s$, the measure $(d^{-s}\delta^{s-k})^nf_n^*(\omega^s)\wedge g_n^*(\omega^{k-s})$ is a probability measure, whence $\mathrm{vol}(L_n)\geq{k \choose s} $ for any $n$, so that $\lim_{n\to\infty}\mathrm{vol}(L_n)={k \choose s} =:\mathrm{V}>0$. 

\medskip

We now show that $\mathrm{vol}(L_n)^{-1}(\psi_n)_*c_1(\bar{L}_n)^k$ converges to the measure $\mu_f=T_f^s\wedge T_{f^{-1}}^{k-s}$. As above, we have
\begin{align*}
(\psi_n)_*c_1(\bar{L}_n)^k &=\sum_{j=0}^k{k \choose j} \frac{1}{d^{nj}\delta^{n(k-j)}}(\psi_n)_*\left(f_n^*(\omega^j)\wedge g_n^*(\omega^{k-j})\right)\\
& = \frac{{k\choose s}}{d^{ns}\delta^{n(k-s)}}(\psi_n)_*\left(f_n^*(\omega^s)\wedge g_n^*(\omega^{k-s})\right)+\nu_n,
\end{align*}
where the mass of $\nu_n$ is $O(\min\{d,\delta\}^{-n})$, whence tends to $0$ as $n\to\infty$. Also
\begin{align*}
\frac{1}{d^{ns}\delta^{n(k-s)}}(\psi_n)_*\left(f_n^*(\omega^s)\wedge g_n^*(\omega^{k-s})\right)=\frac{1}{d^{ns}\delta^{n(k-s)}}(f^n)^*(\omega^s)\wedge (f^{-n})^*(\omega^{k-s}),
\end{align*}
which converges towards $\mu_f$ as $n\to\infty$ by \eqref{eq:Hartogs}. Since $\mathrm{vol}(L_n)\to {k \choose s}$, this gives
\[\lim_{n\to\infty} \mathrm{vol}(L_n)^{-1}(\psi_n)_*c_1(\bar{L}_n)^k=\mu_f,\]
as expected.

%
%%
%%
%%Moreover, since $(f^n)^*\mathscr{O}_{\mathbb{P}^k}(1)=\mathscr{O}_{\mathbb{P}^k}(d^n)$ and $(f^{-n})^*\mathscr{O}_{\mathbb{P}^k}(1)=\mathscr{O}_{\mathbb{P}^k}(\delta^n)$, we have
%%\begin{center}
%%$\frac{1}{d^n}(f_n)^*\mathscr{O}_{\mathbb{P}^k}(1)=\frac{1}{\delta^n}(g_n)^*\mathscr{O}_{\mathbb{P}^k}(1)=\psi_n^*\mathscr{O}_{\mathbb{P}^k}(1)$
%%\end{center}
%%as line bundles on $X_n$. In particular, $L_n=2(\psi_n)^*\mathscr{O}_{\mathbb{P}^k}(1)$ and for any ample line bundle $M=\mathscr{O}_{\mathbb{P}^k}(N)$ with $N\geq1$ on $\mathbb{P}^k$ and any $0\leq j\leq k$, we have
%%\begin{align*}
%%((\psi_n)^*M^j\cdot L_n^{k-j})=2^{k-j}((\psi_n)^*M^j\cdot(\psi_n)^*\mathscr{O}_{\mathbb{P}^k}(1)^{k-j})=N^j2^{k-j}.
%%\end{align*}
%%Note in particular that $\mathrm{vol}(L_n)=(L_n^k)=2^k$ for any $n\geq1$, whence we can set $\mathrm{V}:=2^k$.
%
%
%XXXXXXXXXX
%
%
%Let now $\mu_n:=2^{-k}(\psi_n)_*(d^{-n}(f_n)^*\omega_{\mathrm{FS}}+\delta^{-n}(g_n)^*\omega_{\mathrm{FS}})^k$. To verify condition (3), we need to check that $\mu_n\to \mu_f$ as $n\to\infty$ in the weak sense of measures on $\mathbb{P}^k$. 
%
%XXXXXX
%
%Attention, probleme potentiel de convergence des masses.
%Il faut vérifier tout cela.

\medskip

We now check condition (2).
The metrized line bundle $\bar{M}_n:=\bar{L}_n-2\psi_n^*\bar{L}_0$ is integrable (in the sense of Zhang) with underlying trivial bundle. Fix a place $v\in M_\mathbb{K}$. Let $E_n$ be the exceptional divisor of $\psi_n$. Then the function
\[u_{n,v}:=\varphi_{f,v}\circ\psi_n+\varphi_{f^{-1},v}\circ\psi_n:(X_n\setminus E_n)(\mathbb{C}_v)\to\mathbb{R}\]
extends continuously to $X_{n,v}^{\mathrm{an}}\setminus E_{n,v}^{\mathrm{an}}$, where $\varphi_{f^{\pm 1},v}$ are the functions introduced in \S~\ref{sec:DTV}. Applying \eqref{ineg:trivial} to $f$ and $f^{-1}$, we see that there exists a constant $C_v\geq0$ such that
\[u_{n,v}\leq C_v \quad \text{on} \ \ X_{n,v}^{\mathrm{an}},\]
and there exists a finite set $S\subset M_\mathbb{K}$ such that $C_v=0$ for all $v\notin S$. 

\medskip

We now pick an integer $N\geq1$ and let $\bar{M}$ be the line bundle $M:=\mathscr{O}_{\mathbb{P}^k}(N)$ endowed with an adelic semi-positive continuous metrization. 
Pick any integer $2\leq j\leq k$. Then
\begin{align*}
\left(\psi_n^*(\bar{M})\right)^j\cdot \left(\bar{L}_n\right)^{k+1-j} & =\left(\psi_n^*(\bar{M})\right)^j\cdot \left(\bar{L}_n\right)^{k-j}\cdot\left(\bar{M}_n\right)+2\left(\psi_n^*(\bar{M})\right)^j\cdot \left(\bar{L}_n\right)^{k-j}\cdot\left(\psi_n^*\bar{L}_0\right).
\end{align*}
As $M_n$ is the trivial bundle on $X_n$, we have
\begin{align*}
\left(\psi_n^*(\bar{M})\right)^j\cdot \left(\bar{L}_n\right)^{k-j}\cdot\left(\bar{M}_n\right) & = \sum_{v\in M_\mathbb{K}}n_v\int_{X_{n,v}^\mathrm{an}}u_{n,v}\cdot c_1(\bar{L}_n)_v^{k-j}\wedge c_1((\psi_n)^*\bar{M})^j_v\\
& \leq C_1\cdot ((\psi_n)^*M^j\cdot L_n^{k-j})\leq C_1 \max\{N,2\}^{k},
\end{align*}
where $C_1:=\sum_{v\in M_\mathbb{K}}n_v\cdot C_v\in\mathbb{R}_+$. In particular, the constant $C_2:=C_1 \max\{N,2\}^{k}$ depends only on $c_1(M)$ and
\begin{align*}
\left(\psi_n^*(\bar{M})\right)^j\cdot \left(\bar{L}_n\right)^{k+1-j} \leq C_2+2\left(\psi_n^*(\bar{M})\right)^j\cdot \left(\bar{L}_n\right)^{k-j}\cdot\left(\psi_n^*\bar{L}_0\right).
\end{align*}
In particular, iterating the process and using the projection formula, we deduce
\begin{align*}
\left(\psi_n^*(\bar{M})\right)^j\cdot \left(\bar{L}_n\right)^{k+1-j} & \leq \left(\sum_{\ell = 0}^{k+1-j}2^\ell C_2 \right)+2^{k+1-j}\left(\psi_n^*(\bar{M})\right)^j\cdot \left(\psi_n^*\bar{L}_0\right)^{k+1-j}\\
& \leq \left(\sum_{\ell = 0}^{k+1-j}2^\ell C_2 \right)C_2+2^{k+1-j}\left(\bar{M}\right)^j\cdot \left(\bar{L}_0\right)^{k+1-j}
\end{align*}
The conclusion follows taking $C:=\left(\sum_{\ell = 0}^{k+1}2^\ell C_2 \right)C_2+\max_j 2^{k+1-j}\left(\bar{M}\right)^j\cdot \left(\bar{L}_0\right)^{k+1-j}$.

\medskip

To conclude the proof, we check that periodic points are quasi-small. By construction, we have 
\[h_{\bar{L}_n}=\frac{1}{d^n}h_{\bar{L}_0}\circ f_n+\frac{1}{\delta^n}h_{\bar{L}_0}\circ g_n=\frac{1}{d^n}h\circ f_n+\frac{1}{\delta^n}h\circ g_n +O(\min\{d,\delta\}^{-n})\quad \text{on} \ X_n(\bar{\mathbb{Q}}),\]
by our choice of $\bar{L}_0$. In particular, we have $h_{\bar{L}_n}\geq0$ on $ X_n(\bar{\mathbb{Q}})$ and $h_{\bar{L}_n}(X_n)\geq0$ by, e.g.,~\cite[Lemma~7]{Good-height}. As above, we denote by $\mathbb{P}_f^k(\bar{\mathbb{Q}})$ the set of points with well-defined orbits:
\[\mathbb{P}_f^k(\bar{\mathbb{Q}}):=\mathbb{P}^k(\bar{\mathbb{Q}})\setminus \left(\bigcup_{n\geq0}f^{-n}(I_f)\cup f^n(I_{f^{-1}})\right).\]
Pick any $n\geq1$. By construction of $X_n$ and $\bar{L}_n$, for $x\in\mathbb{P}^k_f(\bar{\mathbb{Q}})$, $x$ avoids the exceptional set of $\psi_n$ for any $n$ and
\begin{align*}
h_{\bar{L}_n}(\psi_n^{-1}(x)) & =\frac{1}{d^n}h_{\bar{L}_0}\circ f_n(\psi_n^{-1}(x))+\frac{1}{\delta^n}h_{\bar{L}_0}\circ g_n(\psi_n^{-1}(x))\\
& =\frac{1}{d^n}h(f^n(x))+\frac{1}{\delta^n}h(f^{-n}(x))+O(\min\{d,\delta\}^{-n}).
\end{align*}
Together with Proposition~\ref{prop:canonical-height}, this implies
\begin{align*}
h_{\bar{L}_n}(\psi_n^{-1}(x))\leq \widehat{h}_f^+(x)+\widehat{h}_f^-(x)+\epsilon_n+O(\min\{d,\delta\}^{-n}), \quad x\in \mathbb{P}_f^k(\bar{\mathbb{Q}}),
\end{align*}
and $h_{\bar{L}_n}(\psi_n^{-1}(x))\leq \epsilon_n+O(\min\{d,\delta\}^{-n})$ for all periodic points $x\in \mathbb{P}^k(\bar{\mathbb{Q}})$ (since $\widehat{h}_f^+(x)=\widehat{h}_f^-(x)=0$ for $x$ periodic). As $\min\{d,\delta\}\geq2$ and $\epsilon_n\to0$, this concludes the proof of the Proposition.
\end{proof}
\subsection{Proof of Theorem~\ref{tm:distrib-periodic} and examples}

Let us explain quickly how to deduce from \cite{Thelin-Nguyen} that the set of isolated periodic points of $f$ is Zariski dense in $\mathbb{P}^k$. By \cite[Theorem~3.4.13]{ThelinVigny1}, we know that the measure $\mu_f$ does not charge (pluripolar hence) strict algebraic sets. We are thus in the settings of \cite[Théorème~5]{Thelin-Nguyen}: isolated hyperbolic periodic points accumulated to a set of measure arbitrarily close to 1 (recurring orbits are of full measure on the natural extension for the lift of $\mu_f$ since it is ergodic) and are thus Zariski dense.

In particular, there exist generic sequences $\{F_i\}_i$ of Galois invariant finite subsets of $\mathbb{P}^k(\bar{\mathbb{Q}})$ of periodic points of $f$ (we do not claim that all points in $\{F_i\}_i$ are hyperbolic). We apply Proposition~\ref{prop:quasi-height} to end the proof of Theorem~\ref{tm:distrib-periodic}. 

\begin{example}\label{examples} \normalfont It is easy to produce examples of birational maps of $\P^k$ defined over a number field $\mathbb{K}$, satisfying \eqref{AS}, \eqref{gooddim} and \eqref{Lee}. To do so, start with a regular automorphism $f$ of $\C^k$ defined of  $\mathbb{K}$, see \cite{Sibony}, which obviously satisfies \eqref{AS}, \eqref{gooddim} and also \eqref{Lee} by \cite[Corollary~C]{Kawaguchi_inequality} or \cite[Theorem~1.2]{Lee_inequality}. Then, for $A\in \mathrm{PSL}(k+1, \mathbb{K})$, $A \circ f$ still satisfies \eqref{gooddim}  and \eqref{Lee}. For $A$ sufficiently close to the identity (at the complex place), $A \circ f$ will still satisfy the improved algebraic stability because it is a \emph{regular birational map} of $\P^k(\C)$: its indeterminacy sets are contained in two disjoint fixed open sets of $\P^k(\C)$, see~\cite{Dinh-Sibony_regular} for a detailed study of such maps. By \cite{Bedford-Diller, ThelinVigny1}, we know that outside a pluripolar set of maps $A \in \mathrm{PSL}(k+1, \C)$, then $A \circ f$  
	satisfies  \eqref{AS} and the energy condition, nevertheless, countable sets are pluripolar so it could be that the $A\in \mathrm{PSL}(k+1, \mathbb{K})$ for which  $A \circ f$ satisfies the improved algebraic stability are exactly those for which $A\circ f$ is a regular birational map.
	
	Let us a give a slight modification of the above construction to show it can produce many birational maps that satisfy the improved algebraic stability. Let us stick to the case of dimension $2$ and degree $2$ for simplicity. Pick $a,b$ in $\mathbb{K}$, a prime number $p$ and a $p$-adic absolute value $|.|_p$ on $\mathbb{K}$ with $|a|_p=|b|_p=1$ and consider the Hénon map
	\[ f([x:y:t])=([x^2+yt+at^2: b xt: t^2]) \ \mathrm{and}\  f^{-1}([x:y:t])=([yt/b: xt- y^2/b^2-at^2: t^2]) \]
	so $I_f= [0:1:0]$ and $I_{f^{-1}}=[1:0:0]$. Take $A \in  \mathrm{PSL}(3, \mathbb{K})$ such that 
	\[A^{-1}([x,y,t])=[a_1x+b_1y+c_1t: a_2x+b_2y+c_2t:a_3x+b_3y+c_3t].\]
with $|b_2|_p$ strictly larger than the $p$-adic absolute values of all the others numbers $a_i$, $b_i$, $c_i$. Then we claim that the map $f \circ A$ is algebraically stable.
	
	For that, we claim by induction on $n$ that, writing  $( f \circ A)^{-n}(A^{-1}(I_f))=[x_n:y_n:t_n]$, then $ |y_n|_p > \max |x_n|_p, |t_n|_p$ which implies the algebraic stability as $ \{(f \circ A)^{-n}I_{f\circ A}\} \cap  \{I_{(f\circ A)^{-1}}\}=\varnothing$. Indeed, observe that $I_{f\circ A}= A^{-1}(I_f)= [b_1:b_2:b_3]$ so the case $n=0$ is clear. Now, assume   $ |y_n|_p > \max |x_n|_p, |t_n|_p$ for some $n$, then 
	\[ [x_{n+1}:y_{n+1}:t_{n+1}]= A^{-1}[y_n t_n/b: x_nt_n-y_n^2/b^2-at_n^2:t_n^2] \]
	and by the strong triangular inequality $|x_nt_n-y_n^2/b^2-at_n^2|_p=|y_n|_p^2$ so applying $A^{-1}$ concludes the induction as $b_2$ dominates all the other coefficients. 

To go further, we show that we can impose the condition that the backward orbit of $I_f$ is Zariski dense. Let us sketch the construction:  assume $|b_2|_p\gg |b_1|_p \gg |b_3|_p \gg \max|a_i|_p,|c_j|_p$ so that $|y_0|_p >2|x_0|_p>4|t_0|_p \neq 0$. An immediate induction then shows that  
\[ |y_n|_p >2^{n+1}|x_n|_p>4^{n+1}|t_n|_p \neq 0.\]
Take any hypersurface over $\Q$. It is given by the equation $P(x,y,t)=0$ for some homogeneous polynomial of degree $p$:
\[ P(x,y,t)= \sum_{i_1+i_2+i_3=p} a_{i_1,i_2,i_3} x^{i_1}y^{i_2}t^{i_3}\]
where the  $a_{i_1,i_2,i_3}$ are in $\bar{\Q}$. Then, it is easy to see that for $n$ large enough, one cannot have $[x_n:y_n:t_n]\stackrel{\forall n}{\in} (P=0)$. 
\end{example}

\bibliographystyle{short}
\bibliography{biblio}
\end{document}